\documentclass[12pt]{amsart}
\usepackage[top=1.25in,bottom=1.05in,left=1in,right=1in]{geometry}

\usepackage{amsmath}
\usepackage{amssymb}
\usepackage{amsthm}
\usepackage{epsfig}
\usepackage{url}
\usepackage{array}
\usepackage{varwidth}
\usepackage{tikz}
\usetikzlibrary{patterns}
\usepackage{multirow}
\usepackage{textcomp}
\usepackage{bbm}
\usepackage{float}
\usepackage{setspace}
\usepackage{color}

\newtheorem{theorem}{Theorem}

\newtheorem{lemma}[theorem]{Lemma}

\newtheorem*{wmattreethm}{The Weighted Matrix-Tree Theorem}
\newtheorem*{matrixdetlemma}{The Matrix Determinant Lemma}

\DeclareMathOperator{\adj}{adj}
\newcommand{\vu}{\mathbf{a}}
\newcommand{\vv}{\mathbf{b}}
\DeclareMathOperator{\indeg}{indeg}
\DeclareMathOperator{\outdeg}{outdeg}

\begin{document}

\title{Linear Algebraic Techniques for Weighted Spanning Tree Enumeration}
\address{Seattle University, Department of Mathematics \\ 901 12th Avenue, Seattle, WA 98122, USA}

\address{Yale-NUS College, Division of Science \\ 10 College Avenue West, \#01-101 Singapore 138609}

\markright{Weighted Spanning Tree Enumeration}
\author{Steven Klee and Matthew T. Stamps}

\maketitle

\begin{abstract}
The weighted spanning tree enumerator of a graph $G$ with weighted edges is the sum of the products of edge weights over all the spanning trees in $G$.  In the special case that all of the edge weights equal $1$, the weighted spanning tree enumerator counts the number of spanning trees in $G$.  The Weighted Matrix-Tree Theorem asserts that the weighted spanning tree enumerator can be calculated from the determinant of a reduced weighted Laplacian matrix of $G$.  That determinant, however, is not always easy to compute.  In this paper, we show how two well-known results from linear algebra, the Matrix Determinant Lemma and the method of Schur complements, can be used to elegantly compute the weighted spanning tree enumerator for several families of graphs.
\end{abstract}



\section{Introduction}

In this paper, we restrict our attention to finite simple graphs.  We will assume each graph has a finite number of vertices and does not contain any loops or multiple edges. We will denote the vertex set and edge set of a graph $G$ by $V(G)$ and $E(G)$, respectively.  A function $\omega: V(G) \times V(G) \rightarrow \mathbb{R}$, whose values are denoted by $\omega_{i,j}$ for $v_i,v_j \in V(G)$, is an \textbf{edge weighting} on $G$ if $\omega_{i,j} = \omega_{j, i}$ for all $v_i, v_j \in V(G)$ and if $\omega_{i,j} = 0$ for all $\{v_i,v_j\} \notin E(G)$.  The \textbf{weighted degree} of a vertex $v_i \in V(G)$ is the value $$\deg(v_i;\omega) = \sum_{v_j \in V(G)} \omega_{i,j} = \sum_{v_j \, : \, \{v_i,v_j\} \in E(G)} \omega_{i,j}.$$
For a matrix $M$, we will use $M(i,j)$ to denote its entry in row $i$ and column $j$ and $M_{i,j}$ to denote the submatrix of $M$ obtained by crossing out its $i^{\text{th}}$ row and $j^{\text{th}}$ column. 

Let $G$ be a graph whose edges are weighted by a function $\omega$.  The \textbf{weighted Laplacian matrix}, $L(G;\omega)$, is a matrix whose rows and columns are indexed by the vertices of $G$ and whose entries are given by
$$
L(G;\omega)(i,j) = 
\begin{cases}
\deg(v_i;\omega) & \text{ if } i = j, \\
-\omega_{i,j} & \text{ if } i \neq j.
\end{cases}
$$

 The weighted Laplacian matrix is singular because its rows sum to zero; however, its cofactors have a beautiful combinatorial interpretation in terms of spanning trees.  

\begin{wmattreethm}[\cite{Kirchhoff,Maxwell}]
Let $G$ be a graph with $V(G) = \{v_1,\ldots,v_n\}$ and let $\omega$ be an edge weighting on $G$.  For any vertices $v_i,v_j \in V(G)$, not necessarily distinct, 
\begin{equation} \label{weighted-mtt-eqn}
\sum_{T \in \mathcal{ST}(G)} \prod_{\{v_k,v_\ell\} \in E(T)}\omega_{k,\ell}  = (-1)^{i+j}\det \left(L(G;\omega)_{i,j}\right),
\end{equation}
where $\mathcal{ST}(G)$ is the set of all spanning trees of $G$.
\end{wmattreethm}

The quantity on the left side of Equation \eqref{weighted-mtt-eqn} is called the \textbf{weighted spanning tree enumerator} for the edge-weighted graph $(G;\omega)$, which we will denote by $\tau(G;\omega)$. Note that when $\omega_{k,\ell} = 1$ for each edge in $G$, $L(G;\omega)$ is the ordinary Laplacian matrix and $\tau(G;\omega)$ counts the number of spanning trees in $G$. Thus, the Weighted Matrix-Tree Theorem is a more general version of the commonly cited Matrix-Tree Theorem.  

As succinct as the Weighted Matrix-Tree Theorem is, computing the weighted spanning tree enumerator from the Laplacian matrix often requires algebraic insights (even for common families of graphs like complete graphs).  For instance, the choice of which row and column to remove can significantly impact the complexity of the resulting determinant. In this paper, we show how two elementary tools from linear algebra -- the Matrix Determinant Lemma and the method of Schur complements --  can be used to elegantly compute the weighted spanning tree enumerator for several well-studied families of graphs without having to eliminate rows or columns.  While the results themselves are not new, we believe the technique is beautiful and worth recording.  Specifically, we give new proofs of the following results: 
\begin{enumerate}
\item the Cayley-Pr\"ufer Theorem, which gives a formula for the weighted spanning tree enumerator for complete graphs \cite{Moon};
\item a generalization of the Cayley-Pr\"ufer Theorem to weighted complete multipartite graphs, a result originally given by Clark \cite{Clark};
\item the weighted spanning tree enumerator for Ferrers graphs, a result originally given by Ehrenborg and van Willigenburg \cite{E-VW}; and
\item the weighted spanning tree enumerator for threshold graphs, a result originally given by Martin and Reiner \cite{Martin-Reiner}.
\end{enumerate}

In all of these cases, the weights on the edges of each graph $G$ have the form $\omega_{i,j} = x_ix_j$ for some indeterminates $\{x_i \ : \ v_i \in V(G)\}$ so the weighted spanning tree enumerator can be viewed as the  polynomial $$\sum_{T \in \mathcal{ST}(G)} \prod_{v_i \in V(G)} x_i^{\deg_T(v_i)}$$ in these variables.  This form is essential because the weights can be encoded in a rank-one matrix that can be added to the weighted Laplacian matrix so that the resulting determinant can be understood with the help of the Matrix Determinant Lemma.  For Ferrers graphs and threshold graphs, we will see that this linear algebraic approach allows us to easily reduce the weighted spanning tree enumerator to the determinant of an upper triangular matrix.

The rest of the paper is structured as follows.  In Section \ref{section:tools} we review the Matrix Determinant Lemma and Schur complement.  We then prove our main result, Lemma \ref{rank-one-update} and immediately show how it can be used to give a quick proof of the Cayley Pr\"ufer theorem, along with a generalization to complete multipartite graphs.  In Section \ref{section:ferrers}, we use Lemma \ref{rank-one-update} to calculate the weighted spanning tree enumerator for Ferrers graphs.  In Section \ref{section:threshold}, we use Lemma \ref{rank-one-update} to calculate the weighted spanning tree enumerator for threshold graphs.  The arguments in Sections \ref{section:ferrers} and \ref{section:threshold} may seem a bit lengthy, but that is simply because we have chosen to carefully describe the entries of several matrices at each step.  The proofs themselves use nothing more than elementary matrix multiplication.  Proofs of the unweighted versions of the results in this paper, which are especially concrete and concise, are presented in \cite{Klee-Stamps-Monthly}.

\section{Tools from linear algebra}\label{section:tools}

In this section, we review two well-known results in linear algebra and present our main result, Lemma \ref{rank-one-update}, whose applicability we will demonstrate in subsequent sections of the paper.  The first linear algebraic result we require is the Matrix Determinant Lemma. Recall that the \textbf{adjugate} of an $n \times n$ matrix is the transpose of its $n \times n$ matrix of cofactors.

\begin{matrixdetlemma}
Let $M$ be an $n \times n$ matrix and let $\vu$ and $\vv$ be column vectors in $\mathbb{R}^n$.  Then $$\det(M+\vu\vv^T) = \det(M) + \vv^T\adj(M)\vu. $$  In particular, if $M$ is invertible, then $\det(M+\vu\vv^T) = \det(M)\left(1+\vv^TM^{-1}\vu\right)$.
\end{matrixdetlemma}

A proof of the Matrix Determinant Lemma can be found in the textbook of Horn and Johnson \cite[\S0.8.5]{Horn-Johnson}, where it is referred to as \emph{Cauchy's formula for the determinant of a rank-one perturbation}.  The main ingredients in the proof are the fact that $\det(\cdot)$ is a multilinear operator on the rows of a matrix and the observation that $\vu\vv^T$ is a rank-one matrix. 

We are now ready to prove our main result.

\begin{lemma} \label{rank-one-update}
Let $G$ be a graph on vertex set $V$ whose edges are weighted by a function $\omega$. Let $L$ be the weighted Laplacian matrix of $G$, and let $\mathbf{a} = (a_i)_{v_i \in V}$ and $ \mathbf{b} = (b_i)_{v_i \in V}$ be column vectors in $\mathbb{R}^V$.  Then $$\det(L + \mathbf{a}\mathbf{b}^T) = \Big(\sum_{v_i \in V} a_i \Big) \cdot \Big(\sum_{v_i \in V} b_i \Big) \cdot \tau(G;\omega).$$
\end{lemma}

\begin{proof}
Let $\mathbf{1}_{V, V}$ denote the $|V| \times |V|$ matrix of ones, and let $\mathbf{1}_V$ denote the $|V| \times 1$ vector of ones. By the Weighted Matrix-Tree Theorem, every cofactor of $L$ equals $ \tau(G;\omega)$, so $$\adj(L) = \tau(G;\omega) \mathbf{1}_{V, V} =  \tau(G;\omega) \mathbf{1}_V \mathbf{1}_V^T.$$ Therefore, by the Matrix Determinant Lemma, 
\begin{eqnarray*}
\det(L + \mathbf{a}\mathbf{b}^T) &=& \det(L) + \mathbf{b}^T \adj(L) \mathbf{a} \\
&=& 0 + \mathbf{b}^T\left(\tau(G;\omega) \mathbf{1}_V \mathbf{1}_V^T\right)\mathbf{a} \\
&=& \left(\mathbf{b}^T\mathbf{1}_V\right)\left( \mathbf{1}_V^T\mathbf{a}\right) \cdot  \tau(G;\omega)\\
&=& \Big(\sum_{v_i \in V} a_i \Big) \cdot \Big(\sum_{v_i \in V} b_i \Big) \cdot  \tau(G;\omega).
\end{eqnarray*}
\end{proof}

In the case that $G$ is unweighted (meaning $\omega_{ij} = 1$ for each edge) and $\vu = \vv = \mathbf{1}_V$, this result simplifies to say $\det(L+\mathbf{1}_{V, V})$ is equal to $|V(G)|^2$ times the number of spanning trees in $G$.  This result was originally given by Temperley \cite{Temperley}, who computed the determinant of $L+\vu\vv^T$ via elementary row operations. 

As an immediate application of this result, we give a simple proof of the Cayley-Pr\"ufer Theorem.

\begin{theorem}
Let $x_1, \ldots,x_n$ be indeterminates and suppose the edges of the complete graph $K_n$ are weighted by $\omega_{ij} = x_ix_j$.  Then the weighted spanning tree enumerator for $K_n$ is given by 
$$
\tau(K_n,\omega)= x_1 \cdots x_n (x_1+\cdots+x_n)^{n-2}.
$$
\end{theorem}

\begin{proof}
Let $L$ denote the weighted Laplacian matrix of $K_n$.  Then the entries of $L$ satisfy 
$$
L(i,j) = \begin{cases}
x_i \cdot \sum_{k \neq i} x_k & \text{ if } i = j, \\
 -x_ix_j & \text{ if } i \neq j
\end{cases}
$$
Consider the vector $\mathbf{x} = (x_1,\ldots,x_n)^T$.  Note that $L + \mathbf{x} \mathbf{x}^T$ is a diagonal matrix whose $i^{\text{th}}$ diagonal entry is $x_i(x_1+\cdots+x_n)$.  Therefore, by Lemma~\ref{rank-one-update}, 
\begin{eqnarray*}
(x_1+\cdots+x_n)^2\cdot \tau(K_n,\omega)  &=& \det(L + \mathbf{x} \mathbf{x}^T) \\
&=& \prod_{i=1}^n x_i(x_1+\cdots+x_n).
\end{eqnarray*}

\end{proof}

This technique can also be applied to calculate weighted spanning tree enumerators for complete multipartite graphs.  Let $G$ be the complete multipartite graph $K_{n_1,\ldots,n_k}$ with its vertex set partitioned as $V_1 \sqcup \cdots \sqcup V_k$.  Let $n = n_1 + \cdots + n_k$ and order the vertices of $G$ so that $v_1, \ldots, v_{n_1}$ are the vertices in $V_1$, $v_{n_1+1}, \ldots, v_{n_1+n_2}$ are the vertices in $V_2$, and so on.  Clark \cite{Clark} originally proved the following result. 

\begin{theorem} \label{complete-multipartite}
Let $G$ be the complete multipartite graph with $n$ vertices and $k$ parts described above, let $x_1, \ldots, x_n$ be indeterminates, and suppose the edges of $G$ are weighted by $\omega_{i,j} = x_ix_j$.  Then the weighted spanning tree enumerator for $G$  is given by 
$$
\tau(G,\omega) = \left( \prod_{i=1}^n x_i \right)  \cdot \Bigg(\prod_{\ell=1}^k \Big(\sum_{v_j \notin V_\ell} x_j\Big)^{n_\ell-1}\Bigg) \cdot \left( \sum_{i=1}^n x_i \right)^{k-2}.
$$
\end{theorem}

\begin{proof}
Once again, let $L$ denote the weighted Laplacian matrix of $G$, let $\mathbf{x} = (x_1,\ldots,x_n)^T$, and consider $L + \mathbf{x}\mathbf{x}^T$, which is a block diagonal matrix with blocks indexed by the sets $V_1, \ldots, V_k$. For each $1 \leq \ell \leq k$, let $\mathbf{x}_\ell$ be the $n_\ell \times 1$ vector of indeterminates corresponding to vertices in $V_\ell$. In other words, $\mathbf{x}_\ell$ is the orthogonal projection of $\mathbf{x}$ onto the subspace spanned by vertices in $V_\ell$. The $\ell^{\text{th}}$ diagonal block of $L + \mathbf{x}\mathbf{x}^T$ has the form $D_\ell + \mathbf{x}_\ell\mathbf{x}_\ell^T$, where $D_\ell$ is the diagonal matrix of weights corresponding to the vertices in $V_\ell$.  Specifically, for a vertex $v_i \in V_\ell$, the corresponding diagonal entry in $D_\ell$ is $\displaystyle x_i \sum_{v_j \notin V_\ell} x_j$. 

By the Matrix Determinant Lemma, 
\begin{eqnarray*}
\det(D_\ell + \mathbf{x}_\ell\mathbf{x}_\ell^T) &=& \det(D_\ell) \left( 1 + \mathbf{x}_\ell^T D_\ell^{-1} \mathbf{x}_\ell \right) \\
&\stackrel{(*)}{=}& \det(D_\ell) \left( 1 + \frac{\sum_{v_i \in V_\ell} x_i}{\sum_{v_j \notin V_\ell} x_j } \right) \\
&=& \det(D_\ell) \left( \frac{\sum_{i=1}^n x_i}{\sum_{v_j \notin V_\ell} x_j} \right) \\
&=& \Big(\prod_{v_i \in V_\ell}x_i\Big) \Big(\sum_{j \notin V_\ell}x_j\Big)^{n_\ell-1} \Big(\sum_{i=1}^n x_i\Big),
\end{eqnarray*}
where the second equality $(*)$ follows from the fact that $D_\ell^{-1}$ can be written as the diagonal matrix of inverse weights $x_i^{-1}$ for $v_i \in V_\ell$ multiplied by $(\sum_{v_j \notin V_\ell}x_j)^{-1}$. The result now follows from Lemma \ref{rank-one-update}, together with the fact that $\det(L+\mathbf{x}\mathbf{x}^T) = \prod_{\ell=1}^k \det(D_\ell + \mathbf{x}_\ell\mathbf{x}_\ell^T)$. 
\end{proof}

Setting $x_i = 1$ for all $1 \leq i \leq n$ in Theorem \ref{complete-multipartite} yields the number of spanning trees in $K_{n_1,\ldots,n_k}$, $$n^{k-2} \cdot \prod_{\ell=1}^k (n-n_\ell)^{n_\ell-1},$$ which has been reproved many times and is originally due to Lewis \cite{Lewis}.

\subsection{The Schur complement of a matrix}

We conclude this section by briefly reviewing the Schur complement of a matrix. Later in this paper, there will be instances in which we partition the vertices of a graph into disjoint subsets as $V(G) = V_1 \sqcup V_2$.  In such instances, the Laplacian matrix of $G$ can be decomposed into a block matrix of the form $$\left( \begin{array}{rr} A & B \\ C & D \end{array} \right),$$ where the first $|V_1|$ rows and columns correspond to the vertices in $V_1$ and the last $|V_2|$ rows and columns correspond to the vertices in $V_2$.  In this case, $A$ and $D$ are square matrices of sizes $|V_1| \times |V_1|$ and $|V_2| \times |V_2|$ respectively. Now suppose $M$ is any square matrix that can be decomposed into blocks $A, B, C, D$ as above with $A$ and $D$ square.  If $D$ is invertible, then the \textbf{Schur complement} of $D$ in $M$ is defined as $M/D:= A- BD^{-1}C$.  A fundamental reason for using Schur complements is the following result (see \cite[\S0.8.5]{Horn-Johnson}).

\begin{lemma} \label{schur-complement}
Let $M$ be a square matrix decomposed into blocks $A, B, C, D$ as above with $A$ and $D$ square and $D$ invertible.  Then $$\det(M) = \det(D) \cdot \det(A-BD^{-1}C).$$
\end{lemma}

When $M = \begin{pmatrix} a & b \\ c & d \end{pmatrix}$ is a $2 \times 2$ matrix and $d \neq 0$, Lemma~\ref{schur-complement} simply says that $\det(M) = d \left(a - b \cdot \frac{1}{d} \cdot c\right),$ which is just an alternate way of writing the familiar formula for the determinant of a $2 \times 2$ matrix.

\section{Weighted spanning tree enumeration for Ferrers graphs} \label{section:ferrers}
In this section, we demonstrate the applicability of Lemmas \ref{rank-one-update} and \ref{schur-complement} for calculating the weighted spanning tree enumerators for Ferrers graphs.

A \textbf{partition} of a positive integer $z$ is an ordered list of positive integers whose sum is $z$.  For example, $(4,4,3,2,1)$ is a partition of $14$.  We write $\lambda = (\lambda_1, \ldots, \lambda_m)$ to denote the parts of the partition $\lambda$.  To any partition, there is an associated bipartite \textbf{Ferrers diagram}, which is a stack of left-justified boxes with $\lambda_1$ boxes in the first row, $\lambda_2$ boxes in the second row, and so on. To any Ferrers diagram there is an associated \textbf{Ferrers graph}, whose vertices are indexed by the rows and columns of the Ferrers diagram with an edge if there is a box in the corresponding position.  The Ferrers diagram and corresponding Ferrers graph associated to the partition $\lambda = (4,4,3,2,1)$ are shown in Figure \ref{example-ferrers}.
\begin{figure}[h]
\begin{center}
\begin{tabular}{>{\centering\arraybackslash}m{.48\textwidth}>{\centering\arraybackslash}m{.48\textwidth}}
\begin{tikzpicture}[scale=.6]
\draw (0,0) -- (4,0);
\draw (0,-1) -- (4,-1);
\draw (0,-2) -- (4,-2);
\draw (0,-3) -- (3,-3);
\draw (0,-4) -- (2,-4);
\draw (0,-5) -- (1,-5);

\draw (0,0) -- (0,-5);
\draw (1,0) -- (1,-5);
\draw (2,0) -- (2,-4);
\draw (3,0) -- (3,-3);
\draw (4,0) -- (4,-2);

\draw (-.5,-.5) node {$r_1$};
\draw (-.5,-1.5) node {$r_2$};
\draw (-.5,-2.5) node {$r_3$};
\draw (-.5,-3.5) node {$r_4$};
\draw (-.5,-4.5) node {$r_5$};

\draw (0.5,0.5) node {$c_1$};
\draw (1.5,0.5) node {$c_2$};
\draw (2.5,0.5) node {$c_3$};
\draw (3.5,0.5) node {$c_4$};
\end{tikzpicture}
&
\begin{tikzpicture}[scale=1.0]
\foreach \t in {(0,0), (1,0), (2,0), (3,0), (4,0), (.5,1.5), (1.5,1.5), (2.5,1.5), (3.5,1.5)}{
	\draw[fill=black] \t circle (.1);
}
\draw (0,0) node[anchor = north] {$r_1$};
\draw (1,0) node[anchor = north] {$r_2$}; 
\draw (2,0)  node[anchor = north] {$r_3$};
\draw (3,0)  node[anchor = north] {$r_4$};
\draw (4,0)  node[anchor = north] {$r_5$};
\draw (.5,1.5)  node[anchor = south] {$c_1$};
\draw (1.5,1.5) node[anchor = south] {$c_2$}; 
\draw (2.5,1.5) node[anchor = south] {$c_3$};
\draw (3.5,1.5) node[anchor = south] {$c_4$};

\draw (0,0) -- (.5,1.5);
\draw (0,0) -- (1.5,1.5);
\draw (0,0) -- (2.5,1.5);
\draw (0,0) -- (3.5,1.5);

\draw (1,0) -- (.5,1.5);
\draw (1,0) -- (1.5,1.5);
\draw (1,0) -- (2.5,1.5);
\draw (1,0) -- (3.5,1.5);

\draw (2,0) -- (.5,1.5);
\draw (2,0) -- (1.5,1.5);
\draw (2,0) -- (2.5,1.5);

\draw (3,0) -- (.5,1.5);
\draw (3,0) -- (1.5,1.5);

\draw (4,0) -- (.5,1.5);
\end{tikzpicture}
\end{tabular}
\end{center}
\caption{The Ferrers diagram (left) and Ferrers graph (right) corresponding to the partition $(4,4,3,2,1)$.}
\label{example-ferrers}
\end{figure}
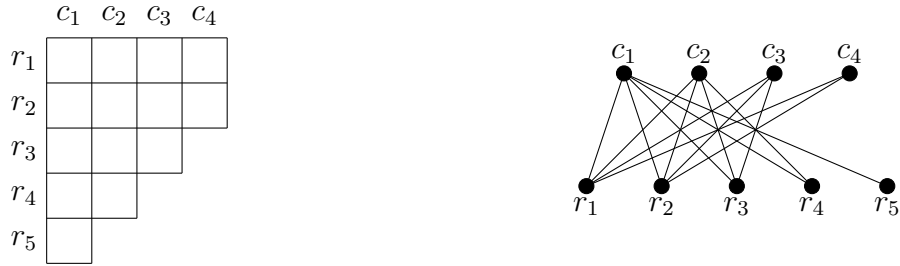
Equivalently, a Ferrers graph is a bipartite graph $G$ whose vertices can be partitioned as $R \sqcup C$ with $R = \{r_1,\ldots,r_m\}$ (corresponding to the rows of a Ferrers diagram) and $C = \{c_1,\ldots,c_n\}$ (corresponding to the columns of a Ferrers diagram) such that 
\begin{enumerate}
\item if $\{r_k,c_{\ell}\} \in E(G)$, then $\{r_i,c_j\} \in E(G)$ for any $i \leq k$ and $j \leq \ell$, and
\item $\{r_1,c_n\} \in E(G)$ and $\{r_m,c_1\} \in E(G)$.
\end{enumerate}

Suppose the edges of a Ferrers graph are weighted by $\omega_{r_i,c_j} = x_iy_j$.  Ehrenborg and van Willigenburg \cite[Theorem 2.1]{E-VW} used the theory of electrical networks to give a closed formula for the weighted spanning tree enumerator in this case. The result is most simply stated using the correspondence between Ferrers graphs and partitions.  If $\lambda = (\lambda_1, \lambda_2, \ldots, \lambda_m)$ is a partition, then its \textbf{conjugate partition} is $\lambda' = (\lambda_1', \lambda_2, \ldots, \lambda_n')$, where $\lambda_j'$ counts the number of parts $\lambda_i$ with $\lambda_i \geq j$ and $n = \lambda_1$ is the size of the largest part in $\lambda$.  In terms of Ferrers diagrams, $\lambda_i$ counts the number of boxes in the $i^{\text{th}}$ row and $\lambda'_j$ counts the number of boxes in the $j^{\text{th}}$ column.

\begin{theorem}\label{thm:ferrers}
Let $G$ be a Ferrers graph whose vertices are partitioned as $V(G) = R \sqcup C$ with $|R| = m$ and $|C| = n$.  Suppose $G$ corresponds to the partition $\lambda = (\lambda_1, \ldots, \lambda_m)$ whose conjugate partition is $\lambda' = (\lambda_1', \lambda_2', \ldots, \lambda_n')$. Let $x_1,\ldots,x_m,$ $y_1,\ldots,y_n$ be indeterminates and suppose the edges of $G$ are weighted by $\omega_{i,j} = x_iy_j$ for each $\{r_i,c_j\} \in E(G)$.  Then 
$$ 
\tau(G; \omega) = \left(\prod_{i = 1}^m x_i\right) \cdot \left(\prod_{j = 1}^n y_j\right) \cdot \left( \prod_{i=2}^m\Big(\sum\limits_{j=1}^{\lambda_i}y_j\Big)\right) \cdot \left(\prod_{j=2}^n\Big(\sum\limits_{i=1}^{\lambda'_j}x_i\Big)\right).
$$

\end{theorem}

\begin{proof}
Let $L$ be the weighted Laplacian matrix of $G$, let $\mathbf{y}$ be the $(m+n) \times 1$ vector given by $\mathbf{y}(r_i) = 0$ for all $1 \leq i \leq m$ and $\mathbf{y}(c_j) = y_j$ for all $1 \leq j \leq n$, and let $\mathbf{x}$ be the $(m+n) \times 1$ vector given by $\mathbf{x}(r_i) = x_i$ for all $1 \leq i \leq m$ and $\mathbf{x}(c_j) = 0$ for all $1 \leq j \leq n$.  (One can think of $\mathbf{y}$ as a weighted indicator vector for $C$ in $\mathbb{R}^V$ and $\mathbf{x}$ as a weighted indicator vector for $R$ in $\mathbb{R}^V$.)  The diagonal entries of $L$ are $x_i \cdot (y_1+\cdots+y_{\lambda_i})$ for each vertex $r_i \in R$ and $y_j \cdot (x_1+\cdots+x_{\lambda'_j})$ for each vertex $c_j \in C$.  We decompose $M:=L + \mathbf{y}\mathbf{x}^T$ as a block matrix 
$$M = \begin{pmatrix} D_R &  B \\ B^{op} &  D_C\end{pmatrix},$$ where 
\begin{itemize}
\item $D_R$ is the diagonal $m \times m$ matrix of weighted degrees of vertices in $R$,
\item $D_C$ is the diagonal $n \times n$ matrix of weighted degrees of vertices in $C$,
\item $B$ is the $m \times n$ matrix with entries $-x_iy_j$ if $\{r_i,c_j\} \in E(G)$ and $0$ otherwise, and 
\item $B^{op} = B^T + (y_1,\ldots,y_n)^T(x_1,\ldots,x_m)$ is the $n \times m$ matrix with entries $x_iy_j$ if $\{r_i,c_j\} \notin E(G)$ and $0$ otherwise. 
\end{itemize}

Let $S:= D_R-BD_C^{-1}B^{op}$ be the Schur complement of the block $D_C$ in $M$. The entries of $S$ can be explicitly computed as follows: Consider rows $r_i$ and $r_j$ that are not necessarily distinct.  The entry $(BD_C^{-1}B^{op})(r_i,r_j)$ is equal to the inner product of the $r_i$-row of $B$ with the $r_j$-column of $D_C^{-1}B^{op}$.  This entry equals $$-\sum\frac{x_ix_jy_k^2}{y_k(x_1+\cdots+x_{\lambda'_k})},$$ where the sum is over all vertices $c_k$ such that $\{r_i,c_k\} \in E(G)$ and $\{r_j,c_k\} \notin E(G)$.  In other words, the sum is over all $c_k \in N(r_i) \setminus N(r_j)$, where $N(\cdot)$ denotes the neighborhood of a vertex. 
Because $G$ is a Ferrers graph, $N(r_1) \supseteq N(r_2) \supseteq \cdots \supseteq N(r_m)$.  This means $(BD_C^{-1}B^{op})(r_i,r_j) = 0$ when $i \geq j$.  Thus $S$ is upper triangular and its diagonal entries are the same as those in $D_R$.  

The result now follows from Lemmas \ref{rank-one-update} and \ref{schur-complement} because 
\begin{eqnarray*}
\sum\limits_{v \in V(G)} \mathbf{y}(v) &=& 
\sum\limits_{j=1}^n y_j = \sum\limits_{j=1}^{\lambda_1} y_j, \\
\sum\limits_{v \in V(G)} \mathbf{x}(v) &=& 
\sum\limits_{i=1}^m x_i = \sum\limits_{i=1}^{\lambda'_1} x_i, \\
\det(D_C) &=& \prod_{j=1}^n \Big(y_j \sum\limits_{i=1}^{\lambda'_j}x_i\Big), \\
\det(S) = \det(D_R) &=& \prod_{i=1}^m \Big(x_i\sum\limits_{j=1}^{\lambda_i} y_j\Big), \text{ and} \\
\det(L+\mathbf{y}\mathbf{x}^T) &=& \det(D_C) \det(S).
\end{eqnarray*}
\end{proof}

\section{Weighted spanning tree enumeration for threshold graphs} \label{section:threshold}

In this section, we demonstrate the applicability of Lemmas \ref{rank-one-update} and \ref{schur-complement} for calculating the spanning tree enumerators of threshold graphs.

Threshold graphs have many equivalent definitions \cite{Mahadev-Peled}, but the following one will be most relevant for us:  A graph $G$ on $n$ vertices is a \textbf{threshold graph} if its vertices can be ordered $v_1,\ldots,v_n$ in such a way that if $\{v_j,v_\ell\} \in E(G)$ for some $1 \leq j < \ell \leq n$, then $\{v_i,v_\ell\} \in E(G)$ for all $i < j$ and $\{v_j,v_k\} \in E(G)$ for all $k < \ell$. An example threshold graph is illustrated in Figure~\ref{example-threshold}.

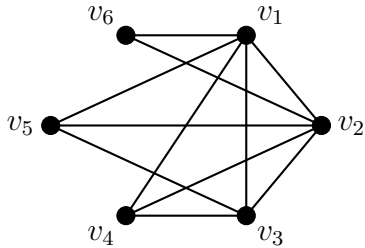
\begin{figure}[h]
\begin{center}
\scalebox{1}{
\begin{tikzpicture}
\draw [fill] (0,2.4) circle [radius=0.12];
\node [above left] at (0,2.4) {$v_6$};
\draw [fill] (1.6,2.4) circle [radius=0.12];
\node [above right] at (1.6,2.4) {$v_1$};
\draw [fill] (2.6,1.2) circle [radius=0.12];
\node [right] at (2.6,1.2) {$ \, v_2$};
\draw [fill] (1.6,0) circle [radius=0.12];
\node [below right] at (1.6,0) {$v_3$};
\draw [fill] (0,0) circle [radius=0.12];
\node [below left] at (0,0) {$v_4$};
\draw [fill] (-1,1.2) circle [radius=0.12];
\node [left] at (-1,1.2) {$v_5 \, $};
\draw [thick] (1.6,2.4) -- (2.6,1.2);
\draw [thick] (1.6,2.4) -- (1.6,0);
\draw [thick] (1.6,2.4) -- (0,0);
\draw [thick] (1.6,2.4) -- (-1,1.2);
\draw [thick] (1.6,2.4) -- (0,2.4);
\draw [thick] (2.6,1.2) -- (1.6,0);
\draw [thick] (2.6,1.2) -- (0,0);
\draw [thick] (2.6,1.2) -- (-1,1.2);
\draw [thick] (2.6,1.2) -- (0,2.4);
\draw [thick] (1.6,0) -- (0,0);
\draw [thick] (1.6,0) -- (-1,1.2);

\end{tikzpicture}}
\end{center}
\caption{An example threshold graph on six vertices.}
\label{example-threshold}
\end{figure} 

If $G$ is a threshold graph, we consider a special vertex $v_t$, where $t$ is the largest index such that $\{v_i,v_j\} \in E(G)$ for all $1 \leq i < j \leq t$. This special vertex in the graph in Figure \ref{example-threshold} is $t = 4$.  For readers who are familiar with the definition that threshold graphs are defined starting from an initial vertex and inductively adding dominating or isolated vertices, $v_{t}$ is the initial vertex, the vertices $v_1,\ldots,v_{t-1}$ are the dominating vertices in $G$ (where $v_{t-1}$ is the first dominating vertex, $v_{t-2}$ the second, and so on), and $v_{t+1}, \ldots, v_n$ are the isolated vertices (where $v_{t+1}$ is the first isolated vertex, $v_{t+2}$ the second, and so on).  Because the vertex degrees in a threshold graph weakly decrease in order, $\deg(v_1) \geq \deg(v_2) \geq \cdots \geq \deg(v_n)$, the degree sequence of a threshold graph is a partition.  

Martin and Reiner \cite{Martin-Reiner} consider a weighting on the edges of a threshold graph in which $\omega_{i,j} = x_{\min(i,j)}y_{\max(i,j)}$ for indeterminates $x_1,\ldots,x_n,y_1,\ldots,y_n$. Because of the natural ordering on the vertices of a threshold graph, the weight of a spanning tree $T$ has a nice interpretation as $$\prod_{i=1}^n x_i^{\indeg_T(v_i)} y_i^{\outdeg_T(v_i)},$$ where the edges of $G$ are oriented from the higher indexed vertex to the lower indexed vertex. With this, Martin and Reiner give the following result \cite[Theorem 4, Eq. (11)]{Martin-Reiner}. 

\begin{theorem}\label{thm:threshold}
Let $G$ be a connected threshold graph with $V(G) = \{v_1,\ldots,v_n\}$ and degree sequence $\delta = (\delta_1,\ldots,\delta_n)$, let $v_t$ be the special vertex defined above, and consider the weighting of edges in $G$ given by $\omega_{i,j} = x_{\min(i,j)}y_{\max(i,j)}$ with indeterminates $x_1,\ldots,x_n,y_1,\ldots,y_n$.  Then 
$$
\tau(G;\omega) = \sum_{T \in \mathcal{ST}(G)} \prod_{i=1}^n x_i^{\indeg_T(v_i)}y_i^{\outdeg_T(v_i)} = x_1 \Big(\prod_{i=t}^n y_i\Big) \Big(\prod_{j=2}^{t-1} f_j\Big) \Big( \prod_{j=t+1}^n g_j \Big) ,
$$
where $$f_j := y_j\sum_{i=1}^jx_i + x_j \sum_{k=j+1}^{1+\delta_j}y_k \qquad \text{ and } \qquad g_j := \sum_{i=1}^{\delta_j}x_i.$$
\end{theorem}

Before we prove this theorem, a few comments are in order.  First, setting $x_i = y_i = 1$ for $1 \leq i \leq n$ tells us that the number of spanning trees in a connected threshold graph $G$ is given by 
$$
\prod_{i=2}^{t-1}(\deg(v_i)+1) \prod_{i=t+1}^n \deg(v_i),
$$
a result that was originally shown by Merris \cite{Merris}. In the particular case that $t=n$, $G$ is the complete graph on $n$ vertices and we recover Cayley's formula: the number of spanning trees in $K_n$ is $\prod_{i=2}^{n-1}n = n^{n-2}$.  Moreover, when $t=n$ and $y_i = x_i$ for all $i$, we get $f_i = x_i(x_1 + \cdots + x_n)$ for all $i$ and recover the Cayley-Pr\"ufer Theorem. 

In their paper, Martin and Reiner describe this formula in terms of an index $s$, which is the side length of the Durfree square in the partition $\delta$.  To translate between their formula and ours, substitute $s = t-1$.

\begin{proof}[Proof of Theorem \ref{thm:threshold}]
Let $L$ be the weighted Laplacian matrix of $G$.  We can partition $L$ as a block matrix of the form 
$$
L = \begin{pmatrix} A & B \\ B^T & D \end{pmatrix},
$$
where the rows and columns of $A$ are indexed by vertices $v_1, \ldots, v_t$ and the rows and columns of $D$ are indexed by vertices $v_{t+1},\ldots,v_n$.

Consider the $n \times 1$ column vectors $\mathbf{y} = (y_1,\ldots,y_n)^T$ and $\mathbf{x} = (x_1,\ldots,x_t,0,\ldots,0)^T$. Because the entry in row $i$ and column $j$ of $\mathbf{y}\mathbf{x}^T$ is $x_jy_i$ if $1 \leq j \leq t$ and zero otherwise, $L+\mathbf{y}\mathbf{x}^T$ has an analogous partition into blocks of the form
$$
L+\mathbf{y}\mathbf{x}^T = \begin{pmatrix} A' & B \\ B^{op} & D \end{pmatrix}.
$$  
As in the proof of Theorem \ref{thm:ferrers}, we can explicitly describe the entries within each block and then compute $\det(L+\mathbf{y}\mathbf{x}^T)$ by taking the Schur complement of block $D$.  
The $j^{\text{th}}$ diagonal entry in block $A$ is $$A(j,j) = \sum_{i=1}^{j-1} x_iy_j + \sum_{k=j+1}^{1+\delta_j}x_jy_k = y_j \sum_{i=1}^{j-1}x_i + x_j \sum_{k=j+1}^{1+\delta_j}y_k.$$ This is because vertex $v_j$ has degree $\delta_j$ and, since $G$ is threshold, its neighbors are the first $\delta_j$ vertices in order other than $v_j$ itself; the weight of an edge $\{v_i,v_j\}$ with $i < j$ is $x_iy_j$; and the weight of an edge $\{v_j,v_k\}$ with $j < k$ is $x_jy_k$. The off-diagonal entries in $A$ are $A(i,j) = -x_iy_j$ if $i<j$ and $A(i,j) = -x_jy_i$ if $i>j$.  Therefore, when adding $\mathbf{y}\mathbf{x}^T$ to $L$, the entries below the diagonal of $A$ vanish.  Thus, $A'$ is upper triangular with diagonal entries $A'(j,j) = x_jy_j + A(j,j) = f_j$.  

Block $B$ records incidences among vertices in $\{v_1,\ldots,v_t\}$ and $\{v_{t+1},\ldots,v_n\}$.  Since the rows in $B$ are indexed by the former set and the columns in $B$ are indexed by the latter, the entries in $B$ are given by $B(i,j) = -x_iy_j$ if $\{v_i,v_j\} \in E(G)$ and zero otherwise. Because of our choice of $\mathbf{y}$ and $\mathbf{x}$, $B^{op}$ is the $(n-t) \times t$ matrix $$B^{op} = B^T + (y_{t+1},\ldots,y_n)^T(x_1,\ldots,x_t).$$  Therefore, $B^{op}(i,j) = x_jy_i$ if $\{v_i,v_j\} \notin E(G)$ and zero otherwise.  
 
 Finally, because the first $t$ vertices of $G$ span a clique and the first $t+1$ vertices do not, we see that $\{v_t,v_{t+1}\} \notin E(G)$.  This implies that $\{v_i,v_j\} \notin E(G)$ for all $t \leq i,j \leq n$.  Therefore, $D$ is a diagonal matrix with diagonal entries $$D(j,j) = y_j \sum_{i=1}^{\delta_j} x_i = y_jg_j$$ for $t+1 \leq j \leq n.$  
 

As in the proof of Theorem \ref{thm:ferrers}, for any $1 \leq i,j \leq t$, the entries of $(BD^{-1}B^{op})(i,j)$ can be written as a weighted sum over vertices $v_k \in N(v_i) \setminus N(v_j)$ such that $t+1 \leq k \leq n$. Because $G$ is threshold, this set of vertices is empty whenever $i \geq j$, so $BD^{-1}B^{op}$ is strictly upper triangular.  Because $A'$ is upper triangular, this means $A'-BD^{-1}B^{op}$ is upper triangular and its diagonal is the same as the diagonal of $A'$. Therefore, by Lemma~\ref{rank-one-update}, 
 
 \begin{eqnarray*}
\Big(\sum\limits_{i=1}^n y_i\Big) \cdot \Big(\sum\limits_{i=1}^t x_i\Big) \cdot \tau(G;\omega) &=& \det(L + \mathbf{y}\mathbf{x}^T)\\
 &=& \det(D) \det(A'-BD^{-1}B^{op}) \\
&=& \Big(\prod_{j=t+1}^n y_jg_j\Big) \Big(\prod_{j=1}^t f_j \Big) \\
 \end{eqnarray*}
 
Since $G$ is connected, $v_n$ is not isolated, which means $\{v_1,v_n\} \in E(G)$.  Because $G$ is threshold, this means $v_1$ is adjacent to every vertex in $G$, so $$f_1 = x_1y_1 + x_1 \sum_{i=2}^n y_i = x_1\sum_{i=1}^n y_i.$$ Similarly, because $v_t$ is only adjacent to $v_1, \ldots, v_{t-1}$, $\delta_{t} = t-1$ and $f_t = y_t\sum_{i=1}^t x_i$. The result follows. 
\end{proof}

\section*{Acknowledgments}

Steven Klee's research was supported by NSF grant DMS-1600048.  Matthew Stamps is grateful to Isabella Novik and the Department of Mathematics at the University of Washington for hosting him during the time this research was conducted.

\bibliographystyle{plain}
\bibliography{bipartitebib}

\begin{thebibliography}{10}

\bibitem{Clark}
Lane Clark.
\newblock On the enumeration of spanning trees of the complete multipartite
  graph.
\newblock {\em Bull. Inst. Combin. Appl.}, 38:50--60, 2003.

\bibitem{E-VW}
Richard Ehrenborg and Stephanie van Willigenburg.
\newblock Enumerative properties of {F}errers graphs.
\newblock {\em Discrete Comput. Geom.}, 32(4):481--492, 2004.

\bibitem{Horn-Johnson}
Roger~A. Horn and Charles~R. Johnson.
\newblock {\em Matrix analysis}.
\newblock Cambridge University Press, Cambridge, second edition, 2013.

\bibitem{Kirchhoff}
G.~Kirchhoff.
\newblock {\"U}ber die {A}ufl\"osung der {G}leichungen, auf welche man bei der
  {U}ntersuchung der linearen {V}erteilung galvanischer {S}tr\"ome gefuhrt
  wird.
\newblock {\em Ann. Phys. Chem.}, 72:497--508, 1847.

\bibitem{Klee-Stamps-Monthly}
S.~Klee and M.~T. Stamps.
\newblock Linear algebraic techniques for spanning tree enumeration.
\newblock {\em Amer. Math. Monthly}, page to appear, 2019.
\newblock (\texttt{arXiv:1903.04973}).

\bibitem{Lewis}
Richard~P. Lewis.
\newblock The number of spanning trees of a complete multipartite graph.
\newblock {\em Discrete Math.}, 197/198:537--541, 1999.
\newblock 16th British Combinatorial Conference (London, 1997).

\bibitem{Mahadev-Peled}
N.~V.~R. Mahadev and U.~N. Peled.
\newblock {\em Threshold graphs and related topics}, volume~56 of {\em Annals
  of Discrete Mathematics}.
\newblock North-Holland Publishing Co., Amsterdam, 1995.

\bibitem{Martin-Reiner}
Jeremy~L. Martin and Victor Reiner.
\newblock Factorization of some weighted spanning tree enumerators.
\newblock {\em J. Combin. Theory Ser. A}, 104(2):287--300, 2003.

\bibitem{Maxwell}
James~Clerk Maxwell.
\newblock {\em A treatise on electricity and magnetism. {V}ol. 1}.
\newblock Oxford Classic Texts in the Physical Sciences. The Clarendon Press,
  Oxford University Press, New York, 1998.
\newblock With prefaces by W. D. Niven and J. J. Thomson, Reprint of the third
  (1891) edition.

\bibitem{Merris}
Russell Merris.
\newblock Degree maximal graphs are {L}aplacian integral.
\newblock {\em Linear Algebra Appl.}, 199:381--389, 1994.

\bibitem{Moon}
J.~W. Moon.
\newblock {\em Counting labelled trees}, volume 1969 of {\em From lectures
  delivered to the Twelfth Biennial Seminar of the Canadian Mathematical
  Congress (Vancouver)}.
\newblock Canadian Mathematical Congress, Montreal, Que., 1970.

\bibitem{Temperley}
H.~N.~V. Temperley.
\newblock On the mutual cancellation of cluster integrals in {M}ayer's fugacity
  series.
\newblock {\em Proc. Phys. Soc.}, 83:3--16, 1964.

\end{thebibliography}

\end{document}